\documentclass[a4paper,11pt]{article}
\usepackage{amsfonts}
\usepackage{amsthm}
\usepackage{amsmath}
\usepackage{latexsym}
\usepackage{amssymb}
\usepackage[ansinew]{inputenc}
\usepackage{graphicx}
\usepackage{dsfont}

\newtheorem*{defn}{Definition}
\newtheorem{fed}{Definition}[section]
\newtheorem*{fed*}{Definition}
\newtheorem*{feds*}{Definitions}
\newtheorem{teo}[fed]{Theorem}
\newtheorem*{teo*}{Theorem}
\newtheorem{lem}[fed]{Lemma}
\newtheorem{cor}[fed]{Corollary}
\newtheorem{pro}[fed]{Proposition}
\theoremstyle{definition}
\newtheorem{rem}[fed]{Remark}

\newtheorem*{rems*}{Remarks}

\def\n0{n_{ \text{\rm \tiny o}}}

\def\bce{\begin{center}}
\def\ece{\end{center}}

\DeclareMathOperator*{\argmin}{argmin}

\def\ds{\displaystyle}

\def\cF{\mathcal F}

\def\QED{\hfill $\square$}

\def\bm{\left[\begin{array}}
\def\em{\end{array}\right]}
\def\ben{\begin{enumerate}}
\def\een{\end{enumerate}}
\def\bit{\begin{itemize}}
\def\eit{\end{itemize}}
\def\barr{\begin{array}}
\def\earr{\end{array}}

\def\eps{\varepsilon}

\def\N{\mathbb{N}}

\def\cV{{\cal V}}

\newcommand{\sub}[2]{{#1}_{\mbox{\tiny{${#2}$}}}}

\newcommand{\mat}{\mathcal{M}_d(\mathbb{C})}

\newcommand{\matsa}{\mathcal{H}(n)}

\newcommand{\matpos}{\mat^+}

\def\beq{\begin{equation}}
\def\eeq{\end{equation}}

\def\Ax2{\,( S_{E(\cF)^\#_\cV})\hat{}_x }

\usepackage{color}

\definecolor{rojo}{rgb}{1,0,0}
\definecolor{azul}{rgb}{0,0,1}

\begin{document}

\title{New deterministic approaches to the least square mean}
\author{ Eduardo M. Ghiglioni 
\footnote{Partially supported by CONICET 
(PIP 0150/14), FONCyT (PICT 1506/15) and  FCE-UNLP (11X681), Argentina. } \ 
 \footnote{ e-mail addresses: eghiglioni@mate.unlp.edu.ar}
\\ {\small Depto. de Matem\'atica, FCE-UNLP
and IAM-CONICET, Argentina  }}

\date{}
\maketitle
\begin{abstract}
	In this paper we presents new deterministic approximations to the least square mean, also called geometric mean or barycenter of a finite collection of positive definite matrices.  Let $A_1, A_2, \ldots, A_m$ be any elements of $\matpos$, where the set $\matpos$ is the open cone in the real vector space of selfadjoint matrices $\matsa$. We consider a sequence of blocks of $m$ matrices, that is, $$(A_1, \ldots, A_m, A_1, \ldots, A_m, A_1, \ldots A_m, \ldots).$$ We take a permutation on every block and then take the usual inductive mean of that new sequence. The main result of this work is that the inductive mean of this block permutation sequence approximate the least square mean on $\matpos$. This generalizes a Theorem obtain by Holbrook. Even more, we have an estimate for the rate of convergence. 
\end{abstract}

\tableofcontents

\section{Introduction}

%--------------------------------------------------------------------------------------------------------------------------------------------------------------
%--------------------------------------------------------------------------------------------------------------------------------------------------------------
%--------------------------------------------------------------------------------------------------------------------------------------------------------------
\subsection{Setting of the problem}
%--------------------------------------------------------------------------------------------------------------------------------------------------------------
%--------------------------------------------------------------------------------------------------------------------------------------------------------------
%--------------------------------------------------------------------------------------------------------------------------------------------------------------

\noindent Let $\matpos$ denote the set of (strictly) positive matrices, which is an open cone in the real vector space of selfadjoint matrices $\matsa$. In particular, it inherits a differential structure where the tangent spaces can be
%. The tangent space to $\matpos$ at any of its points $A$ is the space $T_A\matpos=A\times \matsa$, 
identified with $\matsa$. The manifold $\matpos$ can be endowed with a natural Riemannian structure such that the natural action of the invertible matrices by conjugations becomes isometric. With respect to this metric structure,  if  $\alpha:[a,b]\to\matpos$ is a piecewise smooth path, its length is defined by 
$$
L(\alpha)=\int_a^b \|\alpha^{-1/2}(t)\alpha'(t)\alpha^{-1/2}(t)\|_2 \,dt,
$$
where $\|\cdot\|_2$ denotes the Frobenius or Hilbert-Schmidt norm.
%
%At a point $A\in\matpos$, the metric is given by the differential 
%$$
%ds=\left(\tr(A^{-1}dA)^2\right)^{1/2}=\|A^{-1/2}dA\, A^{-1/2}\|_2,
%$$
%where $\|\cdot\|_2$ denotes the Frobenius or Hilbert-Schmidt norm. Equivalently, if  $\alpha:[a,b]\to\matpos$ is a piecewise smooth path, its length is defined by 
%$$
%L(\alpha)=\int_a^b \|\alpha^{-1/2}(t)\alpha'(t)\alpha^{-1/2}(t)\|_2 \,dt
%$$
%Endowed this metric structure, 
As usual, a distance $\delta$ can be defined by 
$$
\delta(A,B)=\inf \{L(\alpha):\ \mbox{$\alpha$ is a piecewise smooth path connecting $A$ with $B$}\}.
$$
The infimum is actually a minimum, and the geodesic connecting two positive matrices $A$ and $B$ has the following simple expression
$$
\gamma_{AB}(t)=A^{1/2}(A^{-1/2}BA^{-1/2})^{t}A^{1/2}\,.
$$
It is usual in matrix analysis to use the notation $A\#_t B$ instead of $\gamma_{AB}(t)$. The midpoint $A\#_{\frac{1}{2}} B$ is called geometric mean or barycenter between $A$ and $B$. 

\medskip

With the aforementioned Riemannian structure, $\matpos$ becomes a Riemannian manifold with non-positive curvature. In particular, the distance function satisfies the so called semiparallelogram law.

\begin{pro}
	Let $X,, Y$ be two matrices in $\matpos$. Then for any $Z$ in $\matpos$
	\begin{equation}\label{eq1}
	\delta^2(X \#_{\frac{1}{2}} Y, Z) \leq \ds\frac{1}{2}\delta^2(X, Z) + \ds\frac{1}{2}\delta^2(Y, Z) - \ds\frac{1}{4}\delta^2(X, Y).
	\end{equation}
\end{pro}

Inductively for all dyadic rationals $t \in [0, 1]$, and then by continuity we get the following inequality for other points in the geodesic:
\begin{equation}\label{eq2}
\delta^2(X \#_t Y, Z) \leq (1-t)\delta^2(X, Z) + t\delta^2(Y, Z) - t(1-t)\delta^2(X, Y).
\end{equation}
As a consequence of this inequality we get that the function 
$$
f(t) = \delta(A \#_t A^{'}, B \#_t B^{'})
$$
is convex $[0, 1]$ for any $A,A',B,B'\in \matpos$, more precisely 
\begin{equation}\label{eq5}
\delta(A \#_t A^{'}, B \#_t B^{'}) \leq (1-t)\delta(A, B) + t\delta(A^{'}, B^{'}).
\end{equation}

From the semiparallelogram law, we obtain the following alternative characterization of the geometric mean
$$
A\#_{\frac12} B=\argmin_{C\in \matpos} \ \Big(\delta^2(A,C)+\delta^2(B,C)\,\Big).
$$
There is no reason to restrict our attention to only two matrices. The notion of geometric mean can be generalized for more than two matrices in the obvious way
$$
G(A_1,\ldots,A_m):=\argmin_{C\in \matpos} \ \Big(\sum_{j=1}^m\delta^2(A_j,C)\,\Big).
$$
The solution of this least square minimization problem exists and is unique because of the convexity properties of the distance $\delta(\cdot,\cdot)$. A  multivariable version of the semiparallelogram law holds. 

\begin{pro} Let $A_1, A_2, \ldots, A_m$ be any elements of $\matpos$ and let
	$G = G(A_1, \ldots , A_m)$. Then for all $Z \in \matpos$ we have
	\begin{equation}\label{semip}
	\delta^2(Z, G) \leq \ds\frac{1}{m}\left(\ds\sum_{j = 1}^{m} \delta^2(Z, A_j) - \delta^2(G, A_j)\right).
	\end{equation}
\end{pro}
It is also known as the variance inequality, since $G(A_1,\ldots,A_m)$ can be interpreted as a nonlinear expectation of the probability measure
$$
\mu=\frac{1}{m}\sum_{j=1}^m \delta_{\{A_j\}}.
$$

%\noindent	It will we convenient to denote
%$$
%\alpha := \frac{1}{m}\ds\sum_{i = 1}^{m}  \delta^2(G, A_i).
%$$

The geometric means naturally appear in many applied problems. For instance, they appear in the study of radar signals (see \cite{Barba} and the references therein for more details). Another typical application of the geometric means is in problems related with the gradient or Newton like optimization methods (see \cite{Bi-Ia},\cite{Maher}).

%--------------------------------------------------------------------------------------------------------------------------------------------------------------
%--------------------------------------------------------------------------------------------------------------------------------------------------------------
%--------------------------------------------------------------------------------------------------------------------------------------------------------------
\subsection{The problem}
%--------------------------------------------------------------------------------------------------------------------------------------------------------------
%--------------------------------------------------------------------------------------------------------------------------------------------------------------
%--------------------------------------------------------------------------------------------------------------------------------------------------------------

\noindent The usual problem dealing with geometric means is that the geometric mean of three or more matrices does not have in general a closed formula  (see \cite{[Bhatia2]}, \cite{[Holbrook]}, \cite{[LyL]} and \cite{LimPal}). In \cite{[Holbrook]}, Holbrook proved that they can be approximated by the so called inductive means. 

%To motivate the definition of inductive means, note that given a sequence $\{a_n\}_{n\in\N}$ of complex numbers
%\begin{align*}
%\frac{a_1+a_2+a_3}{3}=&\ds\frac{2}{3}\left(\frac{a_1+a_2}{2}\right)+\frac13\, a_3\\
%\vdots&\\
%\frac{a_1+\ldots+a_n}{n}=&\ds\frac{n-1}{n}\left(\frac{a_1+\ldots+a_{n-1}}{n-1}\right)+\frac1n\, a_n.\\
%\end{align*}
%Let $\gamma_{a,b}(t)=t\,b+(1-t)a$, and for a moment allow us to use the notation $a\,\sharp_t\, b=\gamma_{a,b}(t)$.  Then
%\begin{align*}
%\frac{a_1+a_2+a_3}{3}=&(a_1\,\sharp_{\frac12}\, a_2)\,\sharp_{\frac13}\,a_3\\
%\frac{a_1+a_2+a_3+a_4}{4}=&((a_1\,\sharp_{\frac12}\, a_2)\,\sharp_{\frac13}\,a_3)\,\sharp_{\frac14}\,a_4.
%\end{align*}
%and so on and so forth. The segments are the geodesics in the euclidean space. Thus, in our setting, we can replace the segments by the geodesic associated to the Riemannian structure. This is the idea that leads to the definition of the inductive means. 

\begin{defn}
	Given a sequence of (strictly) positive matrices $A = (A_n)_{n\in\N}$, the inductive means are defined as follows:
	\begin{align*}
	S_1(A) & = A_1\\
	S_{n}(A) & = S_{n-1}(A)  \#_{\frac{1}{n}} A_n \ \ \ \ \ (n \geq 2).
	\end{align*}
\end{defn}

Let  $A_0,\ldots,A_{d-1}$ be positive matrices, and define the cyclic sequence $\sub{A}{cyclic} = (A_n)_{n\in\N}$, where for $n\geq d$ we define
$$
A_n=A_k \quad \mbox{if $n\equiv k$ $\mod d$}.
$$

%
% function $F:\Z_d\to\matpos$ by $F(\overline{k})=A_k$, where $\Z_d$ denotes the abelian group of integers mod $d$. Then, if we define the periodic sequence $A=\{F(\overline{n})\}_{n\in\N}$, then Holbrook proved 

Then, the main result in \cite{[Holbrook]} is the following

\begin{teo}[Holbrook]
	$$
	\lim_{n\to\infty} S_{n}(\sub{A}{cyclic}) =G(A_0,\ldots,A_{d-1}).
	$$
\end{teo}

\subsection{The main result of this paper}

%This deterministic approach appear a little while later Bhatia and Karandikar \cite{[Bhatia2]} prove a ``weak law of large number''. In other word, if you take a sequence where every element is taken randomly from $A_0, \ldots, A_{d-1}$ there is a high probability that the inductive mean of this sequence converge to the geometric mean of this matrices. So, we want to find other ways to approximate the geometric mean. This cyclical version made us ask what happens in the case that we permute any of these matrices. The case of finite permutations isn't interesting as in the same work Holbrook prove independence of the starting point. This leaves us to the case of infinite permutations. 

The main advantage of Holbrook's result is that it is deterministic. The results in \cite{[LyL]} and \cite{[Bhatia2]} say that, if we consider the uniform distribution in $\{0,\ldots,d-1\}$ and we construct a sequence 
$$
A_{(r)}=\big(A_{r(1)}, A_{r(2)}, A_{r(3)}, A_{r(4)},\ldots \big)
$$
taking randomly the matrices  $A_0, \ldots, A_{d-1}$, then almost surely 
$$
\lim_{n\to\infty} S_{n}(A_{(r)}) =G(A_0,\ldots,A_{d-1}).
$$
Although this implies that there are plenty of such a sequences, except for the cyclic one $\sub{A}{cyclic}$, we do not know if given a specific sequence $A$ we have that 
\begin{equation}\label{converge}
\lim_{n\to\infty} S_{n}(A) =G(A_0,\ldots,A_{d-1}).
\end{equation}

The main result of this paper allows to enlarge the set of deterministic examples of sequences $A$ such that \eqref{converge} holds. So, fix positive matrices 
$A_0,\ldots,A_{m-1}$, and consider a sequence of permutations of $m$-elements $\sigma=\{\sigma_j\}_{j\in\N}$. Then, define the sequence
$$
\textit{\textbf{A}}_{\sigma} = (A_{\sigma_1(0)},  \ldots, A_{\sigma_1(m-1)}, A_{\sigma_2(0)},  \ldots, A_{\sigma_2(m-1)}, A_{\sigma_3(0)}, \ldots, A_{\sigma_3(m-1)}, \ldots),
$$ 

Using this notation, the main theorem of this paper is the following:

%This result is motivated by the following applied problem. Suppose that we have $n$ sensor, which are sensing a signal. Each sensor codifies the signal through a covariance matrix $A_{k,t}$ where $t$ denotes the time. In this way, for $t=1$ we receive the signals from the sensors in some order
%$$
%A_{\sigma_1(1),1},\ldots,A_{\sigma_1(n),1}.
%$$
%Once a sensor send the information, it has to wait the other sensor for sending again the information. For $t=2$ the sensors send again the information, but perhaps in other order
%$$
%A_{\sigma_2(1),2},\ldots,A_{\sigma_2(n),2}.
%$$
%So, the question is: is it necessary to keep tracking the information of which sensor is sending the information? 

\begin{teo}\label{mainteo}	
	If $n\geq 1$ and $G=G(A_0, \ldots, A_{m-1})$, then there exists $L>0$ depending on these matrices such that
	\begin{equation}
	\delta^2(S_{n}(\textit{\textbf{A}}_{\sigma}), G) \leq \frac{L}{n}.
	\end{equation}
	In particular
	$$
	\lim_{n\to \infty} S_{n}(\textbf{A}_{\sigma})=G. 
	$$
\end{teo}	

%	\noindent	In other words  we considered a sequence of blocks of $m$ matrices $A_1, \ldots, A_m$. We take a permutation on every block and then take the usual inductive mean of that new sequence. The main result of this work is that the inductive mean of this block permutation sequence approximate the least square mean on $\matpos$. Even more, we have an estimate for the convergence,
%	$$
%	\delta^2(S_{km}(\textit{\textbf{A}}_{\sigma}), G) \leq \frac{L}{k},
%	$$
%	for some particular constant $L$.
%	
%\medskip	
%
%\noindent The approximation given by Holbrook is in principle based on the approximation
%of Sturm \cite{[sturm]}, which itself gives convergence to the Karcher mean
%almost surely. 

Holbrook's rate of convergence is $1/n$ as $n \rightarrow \infty$. Actually, as it was pointed out by  Lim and  P\'alfia \cite{[Lim3]}, one cannot expect better convergence rates than $1/n$. For that purpose considerer the case when $m = 2$. So our rate of convergence  should be able to be improved.

\bigskip

Finally, note that the above result is also true if we permutes blocks of length $km$ matrices for some $k\in\N$. More precisely, 
\begin{cor}
	Let $A_1, \ldots, A_m \in \matpos$ be positive matrices. Given $k\geq 1$, consider the sequence 
	$$
	\textbf{A}_{\sigma}^{(k)} = (\underbrace{A_{\sigma_1(1)}, A_{\sigma_1(2)}, \ldots, A_{\sigma_1(m)}, \ldots, A_{\sigma_1(1)}, A_{\sigma_1(2)}, \ldots, A_{\sigma_1(m)}}_{km \ matrices}, A_{\sigma_2(1)}, \ldots)
	$$  
	where each $\sigma$ is a permutation of $km$-elements. Then, there exists $L>0$ such that	
	$$
	\delta^2(S_{n}(\textbf{A}_{\sigma}^{(k)}), G(A_1, \ldots, A_m)) \leq \frac{L}{n}.
	$$ 
\end{cor}

\medskip

Indeed, note that by definition
$$
G(\underbrace{A_1, \ldots, A_m, A_1, \ldots, A_m, \ldots, A_1, \ldots, A_m}_{k \ times}) = G(A_1, \ldots, A_m).
$$

\section{Proof of main result}

In this section we are dedicated to prove Theorem \ref{mainteo}. We begin with some basic fact about the inductive mean which is a direct consequence of \eqref{eq5}.

\begin{lem}\label{conse}
	Given two sequence $A = (A_i)_{i \in \N}, B = (B_i)_{i \in \N}$ in $\matpos$, then
	\begin{equation}\label{desig}
	\delta(S_n(A), S_n(B)) \leq \frac{1}{n}\sum_{i=1}^{n} \delta(A_i, B_i).
	\end{equation}
\end{lem}

\noindent Now we will follow the work of Lim and Pálfia \cite{[Lim2]}. This first lemma is actually step 1 in their paper.

\begin{lem}\label{lemma1}
	Given a sequence $A = (A_i)_{i \in \N}$ in $\matpos$ and $Z\in \matpos$, for every $k,m \in \N$ 
	\begin{align*}
	\delta^2(S_{k+m}(A), Z)  &\leq \ \frac{k}{k+m}\ \delta^2(S_{k}(A), Z) + \ds\frac{1}{k+m}\ds\sum_{j = 0}^{m - 1}\delta^2(A_{k+j+1}, Z)\\ 
	&\quad - \ds\frac{k}{(k+m)^2}\ds\sum_{j = 0}^{m - 1}\delta^2(S_{k+j}(A), A_{k+j+1}). 
	\end{align*}
\end{lem}

\begin{proof}
By the inequality \eqref{eq2} applied to $S_{n+1}(A)=S_n(A)\,\#_{\frac{1}{n+1}}\,A_{n+1}$ we obtain
\begin{align*}
(n+1)\ \delta^2(S_{n+1}(A), Z) - n\ \delta^2(S_{n}(A), Z) & \leq  \delta^2(A_{n+1}, Z) - \ds\frac{n}{(n + 1)}\delta^2(S_{n}(A), A_{n+1}).
\end{align*}
Summing these inequalities from $n=k$ until $n=k+m-1$ we get that the difference 
\begin{align*}
(k+m)\ \delta^2(S_{k+m}(A), Z) - k\ \delta^2(S_{k}(A), Z), 
\end{align*}
obtained from the telescopic sum of the left hand side, is less or equal than 
\begin{align*}
\sum_{j=0}^{m-1}\left(\delta^2(A_{k+j+1}, Z) - \ds\frac{k+j}{(k+j+1)}\delta^2(S_{k+j}(A), A_{k+j+1})\right).
\end{align*}
Finally, using that $\frac{k+j}{k+j+1}\geq \frac{k}{k+m}$ for every $j\in\{0,\ldots,m-1\}$, this sum is bounded from the above by 
\begin{align*}
\sum_{j=0}^{m-1}\left(\delta^2(A_{k+j+1}, Z) - \ds\frac{k}{(k+m)}\delta^2(S_{k+j}(A), A_{k+j+1})\right),
\end{align*}
which completes the proof.
\end{proof}

\noindent	Using the previous result we have this particular case:

\begin{lem}\label{lemma1prima}
	Let $A_1, \ldots, A_m \in \matpos$ fixed. For every sequence $\textbf{A}_{\sigma}$,
	\begin{align*}
	\delta^2(S_{(k+1)m}(\textbf{A}_{\sigma}), G)  & \leq \frac{k}{k+1}\delta^2(S_{km}(\textbf{A}_{\sigma}), G) + \ds\frac{1}{k+1}\left(\frac{1}{m}\ds\sum_{j = 0}^{m - 1}\delta^2((\textbf{A}_{\sigma})_{km+j+1}, G)\right) - \\ \nonumber & - \ds\frac{k}{(k+1)^2}\left(\frac{1}{m}\ds\sum_{j = 0}^{m - 1}\delta^2(S_{km+j}(\textbf{A}_{\sigma}), (\textbf{A}_{\sigma})_{km+j+1})\right). 
	\end{align*}
\end{lem}

\begin{proof}
	Just change $k$ with $km$, $Z$ with $G$ and $A$ with $\textbf{A}_{\sigma}$ in Lemma \ref{lemma1}.

\end{proof}

\noindent	The next step is to find a lower bound for
$$
\frac{1}{m}\ds\sum_{j = 0}^{m - 1}\delta^2(S_{km+j}(\textbf{A}_{\sigma}), (\textbf{A}_{\sigma})_{km+j+1}).
$$
This is step 2 in \cite{[Lim2]} and here it is a little different. 
\\
\\
Let $A_1, \ldots, A_m \in \matpos$ fixed, we will denote
$$
\Delta :=	\max_{1 \leq i, j \leq m} \delta(A_i,A_j); \ \ \ \ \ \ 	\alpha := \frac{1}{m}\ds\sum_{i = 1}^{m}  \delta^2(G, A_i).
$$
Note that by \eqref{eq2}, for every sequence $\textbf{A}_{\sigma}$, all $k \in \N$ and all $j \in \N$, 
$$\delta(S_{km}(\textit{\textbf{A}}_{\sigma} ), (\textbf{A}_{\sigma})_j) \leq \Delta.$$

\begin{lem}\label{lemma2prima}
	Let $A_1, \ldots, A_m \in \matpos$ fixed. For every sequence $\textbf{A}_{\sigma}$ and for all $k \in \N$ we have
	\begin{align*}
	\frac{1}{m} \sum_{j = 0}^{m-1}\delta^2(S_{km+j}(\textbf{A}_{\sigma}), (\textbf{A}_{\sigma})_{km+j+1}) & \geq \delta^2(S_{km}(\textbf{A}_{\sigma}), G) + \alpha   -  \left(\ds\frac{m^2}{(km+1)^2} + 2\ds\frac{m}{km+1}\right) \Delta^2. 
	\end{align*}
\end{lem}

\begin{proof}
	Let $0\leq j \leq m-1$. Note that by \eqref{eq2} and all $k$,
	$$
	\delta(S_{km+j}(\textit{\textbf{A}}_{\sigma}), S_{km+j+1}(\textit{\textbf{A}}_{\sigma})) \leq \frac{\Delta}{km+j+1}.
	$$
	\noindent	Hence
	\begin{align*}
	\delta(S_{km}(\textit{\textbf{A}}_{\sigma}), (\textbf{A}_{\sigma})_{km+j+1}) & \leq  \delta(S_{km}(\textit{\textbf{A}}_{\sigma}), S_{km+j}(\textit{\textbf{A}}_{\sigma})) + \delta(S_{km+j}(\textit{\textbf{A}}_{\sigma}), (\textbf{A}_{\sigma})_{km+j+1}) \\
	& \leq \ds\sum_{h=1}^{j} \ds\frac{\Delta}{km+h} + \delta(S_{km+j}(\textit{\textbf{A}}_{\sigma}), (\textbf{A}_{\sigma})_{km+j+1}) \\
	& \leq  \ds\frac{m}{km+1} \Delta + \delta(S_{km+j}(\textit{\textbf{A}}_{\sigma}), (\textbf{A}_{\sigma})_{km+j+1}). 
	\end{align*}
	\noindent	Therefore, for every $j \leq m$,
	\begin{equation*}
	\delta^2(S_{km}(\textit{\textbf{A}}_{\sigma}), (\textbf{A}_{\sigma})_{km+j+1}) \leq  \left(\ds\frac{m^2}{(km+1)^2} + 2\ds\frac{m}{km+1}\right) \Delta^2 + \delta^2(S_{km+j}(\textit{\textbf{A}}_{\sigma}), (\textbf{A}_{\sigma})_{km+j+1})
	\end{equation*}
	\noindent		where we have used that $\delta(S_{km+j}(\textit{\textbf{A}}_{\sigma}), (\textbf{A}_{\sigma})_{km+j+1}) \leq \Delta$ for every $k,j\in\N$. Summing up these inequalities and dividing by $m$, we get 
	\begin{align}\label{eq8}
	\frac{1}{m} \sum_{j = 0}^{m-1} \delta^2(S_{km}(\textit{\textbf{A}}_{\sigma}), (\textbf{A}_{\sigma})_{km+j+1}) & \leq  \left(\ds\frac{m^2}{(km+1)^2} + 2\ds\frac{m}{km+1}\right) \Delta^2 + \\ \nonumber & + \frac{1}{m} \sum_{j = 0}^{m-1}\delta^2(S_{km+j}(\textit{\textbf{A}}_{\sigma}), (\textbf{A}_{\sigma})_{km+j+1}).
	\end{align}
	By the variance inequality \eqref{semip},
	\begin{equation}\label{comb2}
	\delta^2(S_{km}(\textit{\textbf{A}}_{\sigma}), G) \leq \ds\frac{1}{m}\ds\sum_{i = 1}^{m} \delta^2(S_{km}(\textit{\textbf{A}}_{\sigma}), (\textbf{A}_{\sigma})_i) - \alpha.
	\end{equation}
	Note that
	\begin{equation}\label{clave1}
	\frac{1}{m} \sum_{j = 0}^{m-1} \delta^2(S_{km}(\textit{\textbf{A}}_{\sigma}), (\textbf{A}_{\sigma})_{km+j+1}) = \ds\frac{1}{m}\ds\sum_{i = 1}^{m} \delta^2(S_{km}(\textit{\textbf{A}}_{\sigma}), (\textbf{A}_{\sigma})_i).
	\end{equation}
	So, combining \eqref{eq8} and \eqref{comb2} we get the desired result.		
\end{proof}

\begin{rem}
	In equality \eqref{clave1} is essential that the matrices that appear in the first block $A_{\sigma_1(1)}, \ldots, \\ A_{\sigma_1(m)}$ are the same as those that appear in the $(k+1)$-th block (except by the order). On the other hand, note that this result can not be extended to weighted means as those consider in \cite{[Lim2]}. Indeed, in that setting, the weights do not allow to consider permutations.
\end{rem}

\noindent	Now we prove Theorem \ref{mainteo} and the rate of convergence for a special subsequence.

\begin{lem}\label{subsucesion}
	
	Let $A_1, \ldots, A_m \in \matpos$ fixed. For every sequence $\textbf{A}_{\sigma}$  and for all $k \in \N$,
	$$
	\delta^2(S_{km}(\textit{\textbf{A}}_{\sigma}), G) \leq \frac{L}{k},
	$$
	where $L = \alpha + 3 \Delta^2$. 
\end{lem}

\begin{proof}
	
	We will prove it by induction. If $k = 1$ the result is trivial because $$\delta^2(S_{m}(\textit{\textbf{A}}_{\sigma}), G) \leq \Delta^2 \leq L.$$ Let's suppose that it's true for $k$. Combining Lemmas \ref{lemma1prima} and \ref{lemma2prima} we get
	\begin{align*}
	\delta^2(S_{(k+1)m}(\textit{\textbf{A}}_{\sigma}), G) & \leq \frac{k}{k+1}\delta^2(S_{km}(\textit{\textbf{A}}_{\sigma}), G) + \ds\frac{1}{k+1}\alpha \\
	& - \ds\frac{k}{(k+1)^2}\left[\delta^2(S_{km}(\textit{\textbf{A}}_{\sigma}), G) + \alpha - \left(\ds\frac{m^2}{(km+1)^2} + 2\ds\frac{m}{km+1}\right) \Delta^2\right]  \\
	& \leq \frac{k^2}{(k+1)^2}\delta^2(S_{km}(\textit{\textbf{A}}_{\sigma}), G) + \frac{1}{(k+1)^2}\alpha + 3\frac{1}{(k+1)^2}\Delta^2  \\
	& \leq  \frac{k}{(k+1)^2}L + \frac{1}{(k+1)^2} L  \\
	& \leq \frac{L}{(k+1)}.
	\end{align*}

\end{proof}

\bigskip

\noindent \textit{To conclude the proof of Theorem \ref{mainteo}}. let $n=km+d$ such that $d\in\{1,\ldots, m-1\}$, $k \in \N$. Since $X \#_{t} X = X$ for all $X \in \matpos$ and all $t \in [0, 1]$, using Lemma \ref{conse} with the sequences
\begin{align*}
&(\ (\textbf{A}_{\sigma})_{1},\ldots,(\textbf{A}_{\sigma})_{km}, \underbrace{S_{km}(\textit{\textbf{A}}_{\sigma}),\ldots, S_{km}(\textit{\textbf{A}}_{\sigma})}_{\mbox{\tiny{d times}}}\ )\\
\intertext{and}
&(\ (\textbf{A}_{\sigma})_{1},\ldots,(\textbf{A}_{\sigma})_{km}, \ (\textbf{A}_{\sigma})_{km+1}\ ,\ \ldots\ ,\ (\textbf{A}_{\sigma})_{km+d}\ ),
\end{align*}
and taking into account that $\delta(S_{km}(\textit{\textbf{A}}_{\sigma}),(\textbf{A}_{\sigma})_{km + j})\leq \Delta$ for every $j\in\{1,\ldots,d\}$, we get
\begin{align*}
\delta(S_{km}(\textit{\textbf{A}}_{\sigma}), S_{n}(\textit{\textbf{A}}_{\sigma})) 
& \leq \frac{1}{km + d}\sum_{j = 1}^{d} \delta(S_{km}(\textit{\textbf{A}}_{\sigma}),(\textbf{A}_{\sigma})_{km+j})\\
& \leq \frac{d}{km + d} \Delta 
\leq\frac{1}{k} \Delta\xrightarrow[k \rightarrow \infty]{} 0.
\end{align*}
Combining this with Lemma \ref{subsucesion} we obtain that for $n$ big enough $\delta(S_{n}(\textit{\textbf{A}}_{\sigma}), G)<\eps$.
\QED

\section{Main theorem in Hadamard spaces}

\bigskip

\noindent	In this section we will mention how all the previous result can be generalized to a much more general context. If we look in detail all the previous proofs, we can note that they are based mainly on the semiparallelogram law and the variance inequality (and their consequences). So, with the necessary chances, we can generalized the main theorem to non-positively curved (NPC) spaces, also called Hadamard spaces or (global) CAT(0) spaces.		A complete metric space $(M,\delta)$ is called a Hadamard space if it satisfies the semiparallelogram law, i.e., for each $x, y \in M$ there exists  $m \in M$ satisfying
\begin{equation}\label{desigualdadhadamard}
\delta^2(m, z) \leq \ds\frac{1}{2}\delta^2(x, z) + \ds\frac{1}{2}\delta^2(y, z) - \ds\frac{1}{4}\delta^2(x, y)
\end{equation}
for all $z \in M$.  The point $m$ is called (metric) midpoint between $x$ and $y$. Taking $z=x$ and $z=y$ in the inequality \eqref{desigualdadhadamard}, it is easy to conclude that $\delta(x,m)=\delta(m,y)=\frac12 \delta(x,y)$. Moreover, this inequality also implies that the midpoint is unique. The existence and uniqueness of midpoints give rise to a unique (metric) geodesic $\gamma_{a,b} : [0, 1] \rightarrow M$ connecting any given two points $a$ and $b$, that we denote as before as $a \#_{t} b$. The inequality \eqref{desigualdadhadamard} also extends to arbitrary points on geodesics.
\begin{equation}
\delta^2(x \#_t y, z) \leq (1-t)\delta^2(x, z) + t\delta^2(y, z) - t(1-t)\delta^2(x, y).
\end{equation}

\noindent As in the case of strictly positive matrices, the inductive mean is defined in general Hadamard spaces in the same way (using the geodesic mention before). Also, the notion of barycenter and all the results that we use related to the barycenter (Existence and uniqueness - Variance Inequality) can be extend to this spaces. We refer the reader to \cite{[Palfia]}, \cite{[sturm]}.

\section*{Acknowledgements:}
The author wish to express their gratitude to Professor Jorge A. Antezana and Professor Demetrio Stojanoff fot theirs valuable comments and suggestions.

% ------------------------------------------------------------------------------
% ------------------------------------------------------------------------------

\end{document}